\documentclass[a4paper,12pt, reqno]{amsart}

\usepackage{amssymb, mathtools}
\usepackage{manfnt}
\usepackage{quiver}
\usepackage{tikz-cd}
\usepackage{enumitem}
\usepackage[colorlinks=true, linkcolor=black!10!blue, linkbordercolor=black!15!blue, citecolor=magenta, citebordercolor=magenta, urlcolor=indigo, linktocpage=true]{hyperref}

\setlist[enumerate]{font=\normalfont}

\allowdisplaybreaks
\definecolor{indigo}{HTML}{492DA5}

\newtheorem{thm}{Theorem}
\newtheorem{cor}[thm]{Corollary}
\newtheorem{lemma}[thm]{Lemma}
\newtheorem*{thm*}{Theorem}
\newtheorem{prop}[thm]{Proposition}

\theoremstyle{definition}
\newtheorem{defn}[thm]{Definition}

\theoremstyle{remark}
\newtheorem{rmk}[thm]{Remark}

\newtheorem{examples}[thm]{Examples}

\newcommand{\etale}{{\'e}tale}

\numberwithin{equation}{section}
\numberwithin{thm}{section}

\DeclareMathOperator{\AC}{\bf{AC}}
\DeclareMathOperator{\AComega}{\bf{CC}}
\DeclareMathOperator{\DC}{\bf{DC}}
\DeclareMathOperator{\DMC}{\bf{DMC}}
\DeclareMathOperator{\CCfin}{\bf{CC_{fin}}}

\newcommand{\NN}{\mathbb{N}}
\newcommand{\ZZ}{\mathbb{Z}}
\newcommand{\RR}{\mathbb{R}}
\newcommand{\TT}{\mathbb{T}}

\newcommand{\QQ}{\mathbb{Q}}

\newcommand{\Cst}{\mathrm{C}^*}

\newcommand{\Iso}{\mathrm{Iso}}
\newcommand{\Int}{\mathrm{Int}}

\RequirePackage{color}

\title{Topologically Free non-Hausdorff Groupoids}

\author[Clark]{
    Lisa Orloff Clark$^{\hyperlink{a}{a}}$
}
\author[Thompson]{
   Ryan Thompson$^{\hyperlink{a}{a}, \hyperlink{ast}{\ast}}$
}
\author[Tolich]{
    Ilija Tolich$^{\hyperlink{a}{a}}$ \vspace{2ex} \\ 
    $^{ {\color{black!10!blue}\hypertarget{a}{a}}}$\footnotesize{\textit{S\lowercase{chool of} M\lowercase{athematics and }S\lowercase{tatistics, }V\lowercase{ictoria} U\lowercase{niversity of }W\lowercase{ellington, }PO B\lowercase{ox 600, }W\lowercase{ellington 6140, }NEW ZEALAND}}}

\email[L.O.~Clark]{\href{mailto:lisa.orloffclark@vuw.ac.nz}{lisa.orloffclark@vuw.ac.nz}}
\email[R.~Thompson]{\href{mailto:ryan.thompson@vuw.ac.nz}{ryan.thompson@vuw.ac.nz}}
\email[I.~Tolich]{\href{mailto:ilija.tolich@vuw.ac.nz}{ilija.tolich@vuw.ac.nz}}

\keywords{topologically free groupoid, effective groupoid, topologically principal groupoid, non-Hausdorff groupoid, Axiom of Choice, Baire Category Theorem \\
$^{\hypertarget{ast}{\ast}}$Corresponding author \\
\phantom{$^{\ast}$}Email: \href{mailto:ryan.thompson@vuw.ac.nz}{ryan.thompson@vuw.ac.nz}}
\subjclass{03E25, 22A22, 54E52}

\thanks{This research was supported by the Marsden Fund of the Royal Society of New Zealand (21\nobreakdash-VUW\nobreakdash-156 and 24\nobreakdash-VUW\nobreakdash-014).}

\begin{document}

\begin{abstract}
We study three conditions that control the behaviour of isotropy in étale groupoids, and their relationships under the additional assumptions of second-countability and Hausdorffness. We examine a number of examples that show these properties are distinct. Working under the assumption of the Zermelo-Fraenkel axioms, excluding choice, we then examine an alternate characterization of topological freeness, first introduced by Anantharaman-Delaroche, in the non-Hausdorff setting. Finally, we prove an equivalence between the Baire Category Theorem and an étale groupoid theorem, along with similar equivalences to other weakenings of the Axiom of Choice.
\end{abstract}

\maketitle

\section{Introduction}

\'Etale groupoids have long been an object of study by operator algebraists for their ability to model many important classes of $\Cst$-algebras.  Of particular interest is the behaviour of isotropy -- that is, the collection of morphisms whose source and range coincide. In the \'etale groupoid literature, there exist a number of conditions that control the existence and properties of isotropy in \'etale groupoids: effective, topologically principal and topologically free. The terminology in the literature regarding these properties is inconsistent, and thus one of our goals is to provide an unambiguous reference. For second-countable Hausdorff groupoids, these three conditions coincide (Theorem~\ref{thm:ac}), but in general they are distinct (see Examples~\ref{example}).

As the study of non-Hausdorff groupoids and non-second-countable groupoids in $\Cst$-algebra theory and ring theory has become more common, the nuances of these conditions have become more salient and identifying which condition is the ``correct'' choice is not always obvious and can require additional assumptions (cf. \cite{KM2023} and \cite{Clark2019}).
In this paper we put forward evidence that topologically free is the condition that should be prioritised in the setting of \'etale groupoids. We do so both by gathering results from the literature in an expository fashion, and also by presenting some new supporting results.

We begin in Section~\ref{sec:conditions} by defining the three main conditions,  establishing the relationships between them (see diagram~\ref{dg:diagram}), and briefly discussing related ideas that appear in the literature. We then provide examples to demonstrate that they are indeed distinct (Examples~\ref{example}). In section~\ref{sec:claire}, we provide an equivalent characterisation of the topologically free condition (see Proposition~\ref{prop:claire}). This condition was established in the Hausdorff and second-countable setting by Anantharaman-Delaroches in the 1990's \cite[Proposition~2.3]{Anantharaman-Delaroche1997}, and has proven useful in the more general Hausdorff setting, for example in \cite[Lemma~3.1]{BCFS}.

In the final section, we work only within the Zermelo-Fraenkel (ZF) axioms and consider the relationship between these three groupoid conditions and variants of the Baire Category Theorem -- see Corollaries~\ref{cor1}, \ref{cor2} and \ref{cor3}.  These results improve upon the work of the first author and J.H Brown in \cite{BrownClark}, where a more complicated relationship is established for a special case.  

\section{Preliminaries}

\subsection{Groupoids}

A \emph{groupoid} is a small category $G$ in which every morphism has an inverse. The \emph{unit space} of $G$ is the set of objects of $G$, which we identify with the collection of identity morphisms, and is denoted by $G^{(0)}$. There exist \emph{source} and \emph{range} maps $s, r \colon G \to G^{(0)}$ given by $s(\gamma) \coloneqq \gamma^{-1} \gamma$ and $r(\gamma) \coloneqq \gamma \gamma^{-1}$ for all $\gamma \in G$. The composable pairs of $G$ is the collection $\{(\alpha, \beta) \in G \times G : r(\beta) = s(\alpha)\}$, which we denote by $G^{(2)}$. If $(\alpha, \beta) \in G^{(2)}$ then we write their product as $\alpha \beta$. 

If $U, V \subseteq G$, then we write
\begin{equation}\label{eq:product} UV \coloneqq \{\alpha \beta : (\alpha, \beta) \in U \times V \cap G^{(2)}\}. 
\end{equation}
If $u \in G$, we write $uG$ for $\{u\}G$ and $Gu$ for $G \{u\}$. Note that $uGv = uG \cap Gv$. The \emph{isotropy group at $x$} is $xGx$, and the \emph{isotropy subgroupoid of $G$} is
\[ \Iso(G) \coloneqq \bigcup_{x \in G^{(0)}} xGx. \]
We say that $x \in G^{(0)}$ has \emph{trivial isotropy} if $xGx = \{x\}$ -- otherwise, it is said to have \emph{nontrivial isotropy}. If $V \subseteq G^{(0)}$, then the \emph{restriction of $G$ to $V$}, denoted $G|_V$, is the groupoid $s^{-1}(V) \cap r^{-1}(V)$. A subset $U \subseteq G$ is \emph{invariant} if $s(\gamma) \in U$ implies $r(\gamma) \in U$ for all $\gamma \in G$. 

A \emph{topological groupoid} is a groupoid $G$ endowed with a topology under which the multiplication and inversion maps are continuous. It follows that the range and source maps are also continuous. An open subset $B \subseteq G$ is an \emph{open bisection} if $r|_B$ (equivalently, $s|_B$) are homeomorphisms onto open subsets of $G^{(0)}$ -- in particular, they are injective. A topological groupoid $G$ is \emph{étale} if the range map (equivalently, the source map) is a local homeomorphism. 

It is well-known that if $G$ is étale, then $G^{(0)}$ is open in $G$ (see \cite[Proposition~3.2]{Exel2008}), and each of $xG$ and $Gx$ are discrete with respect to the subspace topology for all $x \in G^{(0)}$ (see \cite[Corollary~2.4.10]{sims}). Moreover, if $G$ is an étale groupoid, then $G$ admits a basis of open bisections  (see \cite[Proposition~3.5]{Exel2008}).  It is useful to note that the collection of all open bisections in and \'etale groupoid is an inverse semigroup such that, for open bisections $U$ and $V$, their product, as defined as in \eqref{eq:product}, is itself an open bisection \cite[Proposition~2.2.3]{patersonGroupoidsInverseSemigroups1999}.

\subsection{Axioms of Choice}

We recall the following forms of choice, and refer the reader to \cite{herrlich2000} for further details. 

\begin{enumerate}[label=(\roman*)]
    \item The \emph{Axiom of Choice} ($\AC$) states that every family of sets admits a choice function. 
    \item The \emph{Axiom of Countable Choice} ($\AComega$) states that every countable family of sets admits a choice function. 
    \item The \emph{Axiom of Countable Finite Choice} ($\CCfin$) states that every countable family of finite sets admits a choice function.
    \item The \emph{Axiom of Dependent Choice} ($\DC$) states that for every nonempty set $X$ and every total relation $R$ on $X$, there exists a sequence $\{x_i\}_{i \in \NN}$ such that $x_i R x_{i+1}$ for all $i \in \NN$. 
    \item The \emph{Axiom of Dependent Multiple Choice} ($\DMC$) states that for every nonempty set $X$ and every total relation $R$ on $X$, there exists a sequence of of nonempty finite sets $\{X_i\}_{i \in \NN}$ such that for all $i \in \NN$ and all $x \in X_i$, there exists $y \in X_{i+1}$ such that $xRy$.
\end{enumerate}

The diagram below outlines the relationship between the choice axioms (see, for instance, \cite{Fossy}). In addition we note that $\CCfin$ and $\DMC$ together imply $\DC$ (\cite{herrlichAxiomChoice2006}).
\[\begin{tikzcd}
	&&&& {\bf{DMC}} \\
	{\bf{AC}} && {\bf{DC}} \\
	&&&& {\bf{CC}}
	\arrow["{\bf{CC_{fin}}}"', shift right=2, Rightarrow, from=1-5, to=2-3]
	\arrow[Rightarrow, from=2-1, to=2-3]
	\arrow[shift right=2, Rightarrow, from=2-3, to=1-5]
	\arrow[Rightarrow, from=2-3, to=3-5]
\end{tikzcd}\]

\section{Groupoid isotropy conditions} \label{sec:conditions}

We now introduce our main isotropy conditions, and briefly discuss related notions that have appeared in the literature.  

\begin{defn}
An \etale\ groupoid $G$ is \emph{effective} if $\Int(\Iso(G))=G^{(0)}$.
\end{defn}

The definitions of effective groupoids found in the literature is fairly consistent. In some cases, the expression ``essentially principal'' has been used to mean effective (see, for example, \cite{exelNonHausdorff} and \cite{matuiHomologyTopologicalFull2012}), but in general essentially principal holds a different meaning. An \etale\ groupoid $G$ is said to be a \emph{strongly effective groupoid} if, whenever $V \subseteq G^{(0)}$ is a closed invariant subset (in the subspace topology), then $G|_V$ is effective. It is clear that if $G$ is strongly effective then it is effective, since $G^{(0)}$ is closed and invariant in itself (cf. \cite{clarkIdealsSteinbergAlgebras2016}). 

In \cite{BrownClark} the first author used the term ``effective'' to refer to the property we now call topologically free. 

\begin{defn}
An \etale\ groupoid $G$ is \emph{topologically principal} if the set of units with trivial isotropy is dense in $G^{(0)}$.
\end{defn}

In \cite[Definition~3.4.6]{EP}, Exel and Pitts use both the expressions ``topologically free'' and ``topologically principal'' interchangeably to refer to topologically principal groupoids. Moreover, they introduce what we refer to as \emph{strongly topologically principal}, which requires that whenever $V \subseteq G^{(0)}$ is closed and invariant (in the subspace topology), then $G|_V$ is topologically principal. Again, it is clear strongly topologically principal implies topologically principal.  

In a purely algebraic sense, a groupoid $G$ is \emph{principal} if $\Iso(G) = G^{(0)}$. It is clear that all principal groupoids are both strongly topologically principal and topologically principal.

\begin{defn}
    An \etale\ groupoid $G$ is \emph{topologically free} if $\Int(\Iso(G)\setminus G^{(0)})=\emptyset$.
\end{defn}

The concept of topological freeness was initially studied by Archbold and Spielberg in the context of actions and dynamical systems (see, for instance, \cite{archboldTopologicallyFreeActions1994}). Kwa\'sniewski and Meyer explicitly defined topologically free as a property of étale groupoids in ~\cite{KM2021}. Later in \cite{KM2023}, they defined an étale groupoid $G$ to be \emph{residually topologically free} if, whenever $V \subseteq G$ is closed and invariant, then the groupoid $G|_V$ is topologically free. To be consistent with the corresponding ``strong'' properties of being effective and topologically principal, we refer to this property as \emph{strongly topologically free}.

Moreover, Kwa\'sniewski and Meyer introduce the notion of an \etale\ groupoid being \emph{AS topologically free}, (AS for Archbold-Spielberg) which requires that whenever $U_1, \ldots, U_n \subseteq G \setminus G^{(0)}$ are finitely many bisections, then the union
\[ \bigcup_{i=1}^{n} \{x \in G^{(0)} : xGx \cap U_i \neq \emptyset\} \]
has empty interior.

We then have the following implications:
\[  \text{Strongly topologically free} \Rightarrow \text{topologically free} \Leftarrow \text{AS topologically free}. \]

The reader is referred to \cite[Remark~2.1]{armstrongReconstructionTwistedSteinberg2023} for additional discussion of these conditions and their relationships. Furthermore, in \cite{steinbergSimplicityInverseSemigroup2021}, the authors discuss these conditions in the context of simplicity of Steinberg algebras, and provide a number of examples of étale groupoids associated to self-similar actions that satisfy some, but not all, of these conditions.

The following relationships between these conditions are established using only ZF axioms.
\begin{thm}
\label{thm:3zf}
Let $G$ be a locally compact \'etale groupoid with Hausdorff unit space.
    \begin{enumerate}
    \item\label{it1:zf} If $G$ is effective, then $G$ is topologically free.
    \item\label{it2:zf} If $G$ is topologically principal, then $G$ is topologically free.
    \item\label{it3:zf} If $G^{(0)}$ is closed in $G$, and $G$ is topologically principal, then $G$ is effective (and topologically free).
\end{enumerate}
\end{thm}

\begin{rmk}
  When $G^{(0)}$ is Hausdorff, the condition that $G^{(0)}$ is closed in $G$ in Theorem~\ref{thm:3zf}~\eqref{it3:zf} is equivalent to $G$ being Hausdorff, see for example \cite[Proposition~3.10]{ExelPardo}. However, we are not assuming  $G^{(0)}$ is Hausdorff. So the condition, as stated, is weaker than $G$ being Hausdorff.  
\end{rmk}

\begin{proof}
    Item~\eqref{it1:zf} follows immediately from the definitions. Item~\eqref{it2:zf} is straightforward contrapositive argument: Suppose $G$ is not topologically free. Then, there exists some open $U \subseteq \Iso(G) \setminus G^{(0)}$, and so $s(U)$ is an open subset of $G^{(0)}$ consisting entirely of units with nontrivial isotropy. Thus, $G$ is not topologically principal.
    
    Finally, item~\eqref{it3:zf} is similar to \cite[Proposition~3.6(i)]{Renault2008}.  Again, we provide the details via a contrapositive argument.  Suppose $G$ is not effective.  Then there exists $U \subseteq \Iso(G)$ such that $U$ is not contained in the unit space.  Thus $U \setminus G^{(0)}$ is a nonempty open subset of $G$ that consists entirely of nontrivial isotropy.  Since $r$ is an open map, $r(U \setminus G^{(0)})$ is an open subset of $G^{(0)}$ consisting entirely of units with nontrivial isotropy.  Therefore, $G$ is not topologically principal.  
\end{proof}

By adding $\AComega$, we have the following theorem. 

\begin{thm}
\label{thm:ac}
Let $G$ be a locally compact \'etale groupoid with Hausdorff unit space.
\begin{enumerate}
    \item\label{it1:zfc} If $G$ is second-countable and effective, then it is topologically principal (and topologically free).
    \item\label{it2:zfc} If $G$ is second-countable and Hausdorff, then all three conditions are equivalent.
    \end{enumerate}
\end{thm}

   \begin{proof}
           Item~\eqref{it1:zfc} follows from \cite[Proposition~3.6(ii)]{Renault2008}.   Note that $\AComega$ is used in the last section of the proof. Item~\eqref{it2:zfc} follows from item~\eqref{it1:zf} and Theorem~\ref{thm:3zf}.
   \end{proof} 
In diagram form, we have established the following:
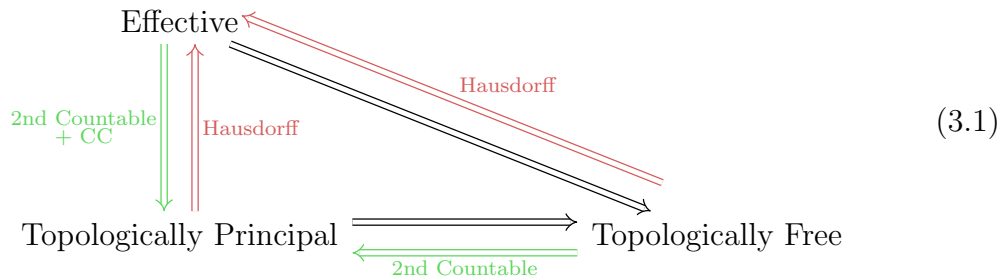
\begin{equation}
\begin{tikzcd}[sep=large] \label{dg:diagram}
	{\text{Effective}} \\
	\\
	{\text{Topologically Principal}} && {\text{Topologically Free}}
	\arrow["\substack{\text{2nd Countable} \\ + \text{ } \mathrm{CC}}"', shift right=2, color={rgb,255:red,92;green,214;blue,92}, Rightarrow, from=1-1, to=3-1]
	\arrow[Rightarrow, from=1-1, to=3-3]
	\arrow["{\text{Hausdorff}}"', shift right=2, color={rgb,255:red,214;green,92;blue,92}, Rightarrow, from=3-1, to=1-1]
	\arrow[shift left=2, Rightarrow, from=3-1, to=3-3]
	\arrow["{\text{Hausdorff}}"', shift right=4, color={rgb,255:red,214;green,92;blue,92}, Rightarrow, from=3-3, to=1-1]
	\arrow["{\text{2nd Countable}}", shift left=2, color={rgb,255:red,92;green,214;blue,92}, Rightarrow, from=3-3, to=3-1]
\end{tikzcd}
\end{equation}

\begin{examples}\label{example}
We gather a number of examples to demonstrate that these three conditions are indeed distinct. The first is well known, see for example, \cite[Example~2.1]{CZ}. 
    \begin{enumerate}
        \item 
        Consider $[0,1]$ with the standard topology $\tau$, and let $p$ be an adjoined point. The \emph{two-headed snake} groupoid is given by the collection 
        \[G \coloneqq [0,1]\cup\{p\}\] 
        equipped with the topology
         \[\tau_G \coloneqq \tau \cup \{U\cup\{p\}-\{1\}\colon U\in \tau \text{ and } 1\in U\}. \]
         The composable pairs are given by
        \[G^{(2)} \coloneqq \{(x,x)\colon x\in[0,1]\}\cup\{(1,p),(p,p),(p,1)\},\]
        and we set $pp = x$. 
        
        Then $G$ is an étale groupoid that is not Hausdorff.  
        It is clear that $G$ is topologically principal since $[0,1)$ consists entirely of trivial isotropy and is dense in $[0,1]$. However, $G$ is not effective since $\Iso(G)=G$.

        \item 
        The following example is from ~\cite[Example 6.3, 6.4]{BCFS}. Let $K$ denote the Cantor set, and define the topological product space $X \coloneqq (K \cap (0, 1)) \times \TT$. We define a continuous action of the real numbers $\RR$, equipped with the discrete topology, on $X$ by
        \[ t \cdot (s, e^{i \theta}) \coloneqq (s, e^{i(\theta + 2st \pi)}). \]
        We denote by $G$ the transformation groupoid of this action with elements $\RR\times (K \cap (0, 1)) \times \TT$ (see \cite{patersonGroupoidsInverseSemigroups1999} for details of this construction). We identify the unit space $G^{(0)}$ with $(K \cap (0, 1)) \times \TT$. Then $G$ is an étale Hausdorff groupoid with totally disconnected unit space, but is not second-countable. 

        For each $x \in G^{(0)}$, the isotropy group $xGx$ consists of elements $(\frac{n}{s}, (x, e^{i \theta}))$ for $n \in \ZZ$, and so every element of the unit space has nontrivial isotropy -- hence, $G$ is not topologically principal. To see that it is effective, let $U \subseteq G$ be open -- we claim that $U \setminus \Iso(G) \neq \emptyset$. Since $U$ is open, there exist $a < b \in (0, 1)$, $\theta \in (0, 2\pi)$ and $t \in \RR \setminus \{0\}$ such that
        \[ \{t\} \times (( (a, b) \cap K ) \times \{e^{i \theta}\})  \subseteq U. \]
        Let $s \in (a, b)$. If $st \notin \ZZ$ then $(t, (s, e^{i \theta})) \in U \setminus \Iso(G)$ and we are done. Suppose $st \in \ZZ$. Then, there exists $\varepsilon \in (0, \frac{1}{t})$ such that $s + \varepsilon \in (a, b)$. Then,
        \[ st < (s+\varepsilon)t < st+1, \]
        and so $(s + \varepsilon)t \notin \ZZ$, and $(t, (s+\varepsilon, e^{i \theta})) \in U \setminus \Iso(G)$. Thus, $G$ is effective. 
        
        \item \label{ex3} 
        The following is a modification of the groupoid construction we use later on in our proof of Theorem~\ref{Baire Category theorem}. Let $X \coloneqq \QQ$, and for each $x \in X$ adjoin a point $\gamma_x$ such that $\gamma_x \gamma_x = x$. Set $G \coloneqq X \cup \{\gamma_x : x \in X\}$. Let $\mathcal{B}$ be a basis for the subspace topology on $X$ inherited from $\RR$, and for each $x \in X$ and $B \in \mathcal{B}$ with $x \in B$ let $B_x \coloneqq (B - \{x\}) \cup \{\gamma_x\}$. Then, the collection
        \[ \mathcal{B} \cup \{B_x : x \in X\} \]
        is basis of open bisections making making $G$ an étale groupoid. Notice that $X$ consists entirely of non-trivial isotropy, and so $G$ is not topologically principal. Similarly, all of $G$ is isotropy, so $G$ is not effective. However, every open subset of $G$ intersects $X$, and so $G$ is topologically free. 
    \end{enumerate}
\end{examples}

\section{Non-Hausdorff Topological Freeness}
\label{sec:claire}
 In this section, we present an alternate characterization of topological freeness. The original statement was presented by C. Anantharaman-Delaroche (see \cite[Proposition~2.3]{Anantharaman-Delaroche1997}), and was associated to a topologically principal groupoid under the additional assumptions that $G$ is both second-countable and Hausdorff.  An equivalence was established for Hausdorff effective groupoids in \cite[Lemma~3.1]{BCFS}. We are not aware of anyone considering the non-Hausdorff case.
When $G$ is not assumed to be Hausdorff, the unit space may not be closed in $G$. This forces us to approach the proof differently to \cite[Proposition~2.3]{Anantharaman-Delaroche1997}

\begin{prop}
\label{prop:claire}
    Let $G$ be a locally compact \etale\ groupoid, and consider the following.
    \begin{enumerate}
     \item\label{it1:claire} For every open $U\subseteq G^{(0)}$ and compact $K\subseteq G \setminus G^{(0)}$ there exists a non-empty open set $V\subseteq U$ such that $VKV=\emptyset$.
        \item\label{it2:claire} The groupoid $G$ is topologically free.
    \end{enumerate}
    Then $\eqref{it1:claire} \implies \eqref{it2:claire}$. If $G^{(0)}$ is Hausdorff, we have $\eqref{it1:claire} \iff \eqref{it2:claire}$.
\end{prop}

\begin{proof}
   For \eqref{it1:claire} implies \eqref{it2:claire},  we prove the contrapositive.  Suppose that $G$ is not topologically free. Let $W  \subseteq \Iso(G)\setminus G^{(0)}$ be a nonempty open set. Since $G$ is \'etale, we may assume that $W$ is a bisection.  Fix $\gamma\in W$. Using that $G$ is locally compact,  we can find an open $B\subseteq W$ containing $\gamma$ such that $B$ is contained in a compact set $K$. Without loss of generality, we can assume that $K$ is contained in $G \setminus G^{(0)}$ -- if not, we replace it with $K \cap (G \setminus G^{(0)})$. Then, since $G$ is \'etale, $G \setminus G^{(0)}$ is closed and so $K \cap (G \setminus G^{(0)})$ is closed in $K$ and thus compact.

   Now consider $r(B)\subseteq G^{(0)}$ and fix $u\in r(B)$. Then $r|_B^{-1}(u)\in B\subseteq K$. Since $B\subseteq \Iso(G)$ we also have $s(r|_B^{-1}(u))=u$. Thus $r|_B^{-1}(u)\in uKu$. Hence, for every nonempty open set $V\subseteq r(B)$ we will have $VKV\neq \emptyset$.

     For the reverse implication, assume that $G^{(0)}$ is Hausdorff, and suppose that $G$ is topologically free.
     Fix a compact $K\subseteq G\setminus G^{(0)}$ and an open $U\subseteq G^{(0)}$. Since $K$ is compact, there exist finitely many open bisections $\{B_i\}_{i=0}^n$ such that $K \subseteq \bigcup_{i=0}^n B_i$. We can further assume that $K\cap B_i\neq \emptyset$ for each $i \in \{0, \ldots, n\}$.    We proceed by induction on the the number of bisections required to cover $K$.

    For the base case, let $B$ be an open bisection such that $K \subseteq B$.
    If $UKU=\emptyset$, then setting $V \coloneqq U$, we are done.  Suppose $UKU\neq \emptyset$.  We may also assume that $U\cap B=\emptyset$, otherwise $V\coloneqq U\cap B$ would satisfy $VKV=\emptyset$. Thus, under these assumptions, $UBU \subseteq G \setminus G^{(0)}$ is a non-empty open set.

    Since $G$ is topologically free, we have that $\Int(\Iso(G)\setminus G^{(0)})=\emptyset$, and so the open set $UBU$ is not entirely isotropy. Thus, there exists $\gamma \in UBU \setminus \Iso(G)$ such that $u:=r(\gamma)\in U$. 
    
    We claim that there exists a neighbourhood $V$ of $u$ such that $VBV=\emptyset$. Let $\{V_j\}_{j \in J}$ be a neighbourhood basis for $u$ contained in $U$. Toward a contradiction, suppose that for every $j \in J$, there exists some $\alpha_j \in V_j B V_j$ -- that is, $r(\alpha_j), s(\alpha_j) \in V_j$. It follows that $r(\alpha_j) \to u$, and so by \cite[Proposition~1.15]{DanaCrossed}, there exists a subnet $\{\alpha_{j_l}\}$ contained in $B$ such that $\alpha_{j_l} \to s|_{B}^{-1}(u) = \gamma$. Since $r$ is continuous, and $G^{(0)}$ is Hausdorff, we have $r(s^{-1}|_B(u)) = s(s^{-1}|_B(u)) = u$. Since $B$ is a bisection, it must be that $\gamma = s^{-1}|_B(u)$, and so $\gamma \in \Iso(G)$, a contradiction. Thus, there is some neighbourhood $V$ of $u$ satisfying $VBV = \emptyset$.

   For the inductive hypothesis,  suppose that there is $m > 1$ such that for every open set $U\subseteq G^{(0)}$ and every compact set $K\subseteq G\setminus G^{(0)}$ such that there exist $m$ bisections $\{B_i\}_{i=1}^{m}$ that cover $K$,  
   there exists a non-empty open set $V\subseteq U$ such that $VKV=\emptyset.$
   
   Fix a compact set $K\subseteq G \setminus G^{(0)}$ and let $\{B_i\}_{i=1}^{m+1}$ be a collection of $m+1$ bisections which cover $K$.
   We partition $K$ into two compact components -- defining $$K_{m+1}:=K\setminus \cup^m_{i=1}B_i, \qquad K_m \coloneqq K \setminus B_{m+1}$$ we see that $K = K_{m+1} \cup K_m$ and $K_{m+1}$ is compact and coverable by $B_{m+1}$ alone. Hence, there exists a nonempty open set $U'\subseteq U$ such that $U'K_{m+1}U'=\emptyset.$
   
   Moreover, $K_{m} \coloneqq K\setminus B_{m+1}$ is compact and coverable by the $m$ bisections $\{B_i\}_{i=1}^{m}$. Thus, there exists a non-empty open set $V\subseteq U'\subseteq U$ such that $VK_{m}V=\emptyset$. Since $V\subseteq U'$ and $U'K_{m+1}U'=\emptyset$, it follows that $VKV=\emptyset$.

\end{proof}

\section{Analogue of the Baire Category Theorem}

A topological space has the \emph{Baire} property if any countable union of closed sets with empty interior also has empty interior (see \cite[Chapter~8]{munkresTopology2000}). In this section, working without $\AC$, we establish a connection between Baire spaces and the relationship between the topologically free and topologically principal properties of \'etale groupoids. As corollaries, we extend this connection to various weakenings of the $\AC$, including the $\AComega$ and $\DC$. 

A special case was examined by Brown and Clark in \cite{BrownClark}. Working in the context of a (not necessarily étale) topological groupoid with additional assumptions, a particular formulation of the Baire category theorem was shown to be equivalent to a statement relating being effective and topologically principal. In particular, the formulation of the Baire category theorem for complete metric spaces used is equivalent to $\DC$ \cite{goldblattRoleBaireCategory1985}. The hypotheses of the result in \cite{BrownClark} are somewhat technical and not very natural -- they prioritized injectivity of the range and source maps at the cost of openness.  In what follows,  we sacrifice injectivity in favour of openness. 

\begin{defn}
    Let $G$ be a topological groupoid. We call an open set $B\subseteq G$ an \emph{open iso-section} if 
$r|_{B \setminus G^{(0)}}$ is injective.
\end{defn}

If $U \subseteq G^{(0)}$, then we write $\overline{U}^0$ to denote the closure of $U$ in $G^{(0)}$ with respect to the subspace topology. If $G$ is Hausdorff then $\overline{U}^0 = \overline{U}$, but in general they need not coincide. The following Lemma is an adaptation of \cite[Lemma 4.3]{BrownClark}.

\begin{lemma}\label{range of isosection}
    Let $G$ be a topologically free \etale\ groupoid, and let $B \subseteq G$ be an open iso-section. Then, 
    \begin{enumerate} 
    \item\label{it1:range} $r(\Iso(B)\setminus G^{(0)})$ has empty interior;
    \item\label{it2:range}$\overline{r(\Iso(B)\setminus G^{(0)})}^0$ has empty interior.  
    \end{enumerate}
\end{lemma}

\begin{proof}  
    For item \eqref{it1:range}, we argue by contradiction. Suppose that there exists an open set $W\subseteq r(\Iso(B)\setminus G^{(0)})$. Then 
    \[r^{-1}(W) \cap B \subseteq \Iso(B) \setminus G^{(0)} \subseteq \Iso(G) \setminus G^{(0)}.\] 
    Since $r$ is continuous and $B$ is open, $r^{-1}(W) \cap B$ is open, which contradicts that $G$ is topologically free.

    For item \eqref{it2:range},
    first we claim that $r(\Iso(B) \setminus G^{((0)})$ is closed in $r(B)$ with respect to the subspace topology. 
     Since $G$ is a topological groupoid, $\Iso(G)$ is closed in $G$ and since $G$ is \'etale, $G^{(0)}$ is open in $G$. We have
     \[
     \Iso(B) \setminus G^{(0)} = B \cap (\Iso(G) \setminus G^{(0)}) 
     \] 
and hence  $\Iso(B) \setminus G^{(0)}$ is closed in $B$.  
     Since $r$ is injective on $\Iso(B)\setminus G^{(0)}$,  it is straightforward to check that $r(\Iso(B) \setminus G^{(0)})$ is closed in $r(B)$ as claimed.  Thus  
     \[\overline{r(\Iso(B)\setminus G^{(0)})}^0 \cap r(B) = r(\Iso(B) \setminus G^{(0)}).\]
     
     We now prove \eqref{it2:range} by way of contradiction. 
 Suppose there exists a nonempty open set 
     \[U \subseteq \overline{r(\Iso(B)\setminus G^{(0)})}^0.
     \]
     Then $U \cap r(B)$ is open because $r$ is an open map, and contains $U \cap r(\Iso(B) \setminus G^{(0)})$ so is nonempty.  But,
\[U \cap r(B) \subseteq \overline{r(\Iso(B)\setminus G^{(0)})}^0  \cap r(B)=r(\Iso(B) \setminus G^{(0)}), \] which contradicts item~\eqref{it1:range} and hence no such $U$ exists.  
\end{proof}

\begin{thm}\label{Baire Category theorem}
    Let $X$ be a topological space. The following are equivalent:
    \begin{enumerate}
    \item\label{it1:baire} The topological space $X$ is Baire.
    \item\label{it2:baire} Let $G$ be an \etale\ groupoid with $G^{(0)}=X$, and suppose $G$ admits a countable cover of open iso-sections. If $G$ is topologically free, then $G$ is topologically principal.
    \end{enumerate}
\end{thm}

\begin{proof}
     For $\eqref{it1:baire}\Rightarrow \eqref{it2:baire}$, the argument is similar to that of \cite[Proposition~2.24]{KM2021} and \cite[Proposition~3.6]{Renault2008}.  Suppose that $X$ is Baire and let $G$ be a topologically free, \etale\ groupoid with $G^{(0)}=X$. Let $\{B_n\}$ be a countable cover of $G$ by open iso-sections. We show that the set of units of $G$ with trivial isotropy is dense in $G^{0}=X$. 
     
     For each $n \in \NN$ let 
     \[C_n:=\overline{r(\Iso(B_n)\setminus G^{(0)})}^0 \subseteq G^{(0)}.\] 
     Then each $C_n$ has empty interior by Lemma~\ref{range of isosection}~\eqref{it2:range}.  Since $X$ is Baire, the set 
     $C\coloneqq \bigcup_{n \in \NN}C_n$
     also has empty interior.  Thus every open set in $G^{(0)}$ has nontrivial intersection with the complement of $C$.   Since $\{B_n\}$ is a cover, all units with nontrivial istropy are contained in $C$ and hence each of the units in the complement of $C$ have trivial isotropy. Thus, the units with trivial isotropy are dense.

     To show \eqref{it2:baire} implies \eqref{it1:baire},  fix a topological space $X$ with a basis $\mathcal{B}$, and a countable collection of closed sets $\{C_n\}$ each with empty interior.  We construct a groupoid $G$ as in \cite[Example~3.2]{BrownClark} that is a generalised two-headed snake where we attach a head for each element of $C_n$.  In particular, we set $G^{(0)} \coloneqq X$ and for each $x \in C_n$ we adjoin a groupoid element $\gamma_x$ such that $s(\gamma_x)=r(\gamma_x)$ and $\gamma_x\gamma_x=x$.  

%
%
%
%

     To define the topology on $G$, for each $x \in X$ and each $B\in \mathcal{B}$ with $x \in B$, we define $B_x=(B -\{x\}) \cup \{\gamma_x\}$, giving us a basis 
     \[
     \mathcal{B} \cup \left( \bigcup_{x\in X}\bigcup_{\{B \in \mathcal{B} : x\in B\}} \{B_x\}\right)
     \]
     under which $G$ is a topological groupoid. Since each basic open set is a bisection, $G$ is \'etale. Without loss of generality, we assume that $X\in \mathcal{B}$ (if not, we may add it without altering the topology on $X$). For each $n \in \NN$ define \[B_n = \bigcup_{x\in C_n} X_x.\] It is clear that $r|_{B_n\setminus G^{(0)}}$ is injective, and as a union of basis elements it is open, hence $B_n$ is an open iso-section. Thus, the collection $\{B_n\} \cup \{X\}$ is a countable cover of $G$ by open iso-sections.  
     
     Notice that every open set in $G$ has nontrivial intersection with $G^{(0)}$, thus $G$ is topologically free and hence \eqref{it2:baire} gives us that $G$ is also topologically principal.  
     That is, the set of units with trivial isotropy is dense in $G^{(0)}$, which in this case, says the complement of $C\coloneqq\cup_n C_n$ is dense.  Therefore $C$ has empty interior and hence $X$ is Baire. 
\end{proof}

\begin{rmk}
    Since an open bisection is automatically an open iso-section,  any second-countable \'etale groupoid has a countable cover of open iso-sections.  However, if an \'etale groupoid has a countable cover of open iso-sections, this does not imply the groupoid is second-countable.  For example, construct a groupoid as above such that $C_n$ is uncountable for some $n \in \NN$. Then, this groupoid has a countable cover of open iso-sections but is not second-countable.  
\end{rmk}

\subsection{Baire Category Theorems}
Since no form of choice is required in the proof of Theorem~\ref{Baire Category theorem}, we establish new equivalences with various formulations of the Axiom of Choice. In each Corollary, Theorem~\ref{Baire Category theorem} gives the equivalence of items \eqref{it1:1} and \eqref{it3:1}.
\begin{cor}
\label{cor1}
The following are equivalent:
\begin{enumerate}
	\item\label{it1:1} every second-countable complete pseudometric space is Baire,
	\item\label{it2:1} $\AComega$, 
	\item\label{it3:1} let $G$ be an \etale\ groupoid that has a countable cover of open iso-sections, such that $G^{(0)}$ is a \underline{second-countable complete pseudometric space}. If $G$ is topologically free, it is topologically principal. 
\end{enumerate}
\end{cor}

\begin{proof}The equivalence of \eqref{it1:1} and  \eqref{it2:1} is shown in \cite[Theorem~3.4]{BH}. 
\end{proof}

\begin{cor}
\label{cor2}
The following are equivalent:
\begin{enumerate}
	\item\label{it1:2} every complete metric space is Baire,
	\item\label{it2:2} $\DC$,
	\item\label{it3:2} let $G$ be an \etale\ groupoid with a countable cover of open iso-sections, such that $G^{(0)}$ is a \underline{complete metric space}. If $G$ is topologically free, it is topologically principal. 
\end{enumerate}
\end{cor}

\begin{proof}The equivalence of \eqref{it1:2} and  \eqref{it2:2} is shown in \cite{Blair}.
\end{proof}

\begin{cor}
\label{cor3}
The following are equivalent:
\begin{enumerate}
	\item\label{it1:3} every compact Hausdorff space is Baire,
	\item\label{it2:3} $\DMC$,
	\item\label{it3:3} let $G$ be an \etale\ groupoid with a countable cover of open iso-sections, such that $G^{(0)}$ is a \underline{compact Hausdorff space}. If $G$ is topologically free, it is topologically principal. 
\end{enumerate}
\end{cor}

\begin{proof}The equivalence of \eqref{it1:3} and  \eqref{it2:3} is shown in \cite[Corollary~3]{Fossy}.
\end{proof}

\begin{rmk}
    In general, having a countable cover of open iso-sections is a weaker condition than having a countable cover of open bisections. However, if $G^{(0)}$ is closed in $G$, then the two conditions are equivalent. In particular, if $\{B_i\}_{i \in \NN}$ is a countable cover of $G$ by iso-sections, we can set $D_i \coloneqq B_i \setminus G^{(0)}$ for each $i \in \NN$. Then, each $D_i$ is an open bisection, and so $\{D_i\}_{i \in \NN} \cup \{X\}$ is a countable cover of $G$ by open bisections. 
\end{rmk}

\bibliographystyle{plain}
\bibliography{library}
\end{document}